\documentclass[12pt]{amsart}

\newtheorem{theorem}{Theorem}[section]
\newtheorem{proposition}[theorem]{Proposition}
\newtheorem{lemma}[theorem]{Lemma}
\newtheorem{corollary}[theorem]{Corollary}

\theoremstyle{definition}
\newtheorem{definition}[theorem]{Definition}

\newtheorem{example}[theorem]{Example}

\newtheorem{conjecture}[theorem]{Conjecture}
\newtheorem{remark}[theorem]{Remark}

\newcommand{\ZZ}{ \ensuremath{\mathbb{Z}}}

\def\cocoa{{\hbox{\rm C\kern-.13em o\kern-.07em C\kern-.13em o\kern-.15em A}}}

\newcommand{\cd}{\textbf{cd}}
\newcommand{\cc}{\mathbf{c}}
\newcommand{\dd}{\mathbf{d}}
\newcommand{\aaa}{\mathbf{a}}
\newcommand{\bb}{\mathbf{b}}
\newcommand{\F}{\mathcal{F}}
\newcommand{\sd}{\mathrm{sd}}
\newcommand{\A}{\mathcal{A}}
\newcommand{\B}{\mathcal{B}}

\begin{document}

\title[On the \cd-index and $\gamma$-vector of S*-shellable CW-spheres]
{On the \cd-index and $\gamma$-vector of \\
S*-shellable CW-spheres}

\author{Satoshi Murai}
\address{
Satoshi Murai,
Department of Mathematical Science,
Faculty of Science,
Yamaguchi University,
1677-1 Yoshida, Yamaguchi 753-8512, Japan
}

\author{Eran Nevo}
\address{
Department of Mathematics,
Ben Gurion University of the Negev,
Be'er Sheva 84105, Israel
}
\email{nevoe@math.bgu.ac.il}

\thanks{
Research of the first author was partially supported by KAKENHI 22740018.
Research of the second author was partially supported by an NSF Award DMS-0757828.
}


\begin{abstract}
We show that the $\gamma$-vector of the order complex of any polytope is the $f$-vector of a balanced simplicial complex. This is done by proving this statement for a subclass of Stanley's S-shellable spheres which includes all polytopes. The proof shows that certain parts of the $\cd$-index, when specializing $\cc=1$ and considering the resulted polynomial in $\dd$, are the $f$-polynomials of  simplicial complexes that can be colored with ``few" colors. We conjecture that the $\cc\dd$-index of a regular CW-sphere is itself the \emph{flag} $f$-vector of a colored simplicial complex in a certain sense.
\end{abstract}

\maketitle

\section{Introduction}
Let $P$ be an ($n-1$)-dimensional regular CW-sphere (that is, a regular CW-complex which is homeomorphic to an ($n-1$)-dimensional sphere).
In face enumeration,
one of the most important combinatorial invariants of $P$
is the \cd-index.
The \cd-index $\Phi_P(\cc,\dd)$ of $P$ is a non-commutative polynomial in the variables
$\cc$ and $\dd$ that encodes the flag $f$-vector of $P$.
By the result of Stanley \cite{St1} and Karu \cite{Ka},
it is known that the \cd-index $\Phi_P(\cc,\dd)$ has non-negative integer coefficients.
On the other hand,
a characterization of the possible $\cd$-indices for regular CW-spheres, or other related families, e.g Gorenstien$^*$ posets, is still beyond reach.
In this paper we take a step in this direction and establish some  non-trivial upper bounds, as we detail now.

If we substitute $1$ for $\cc$ in $\Phi_P(\cc,\dd)$, we obtain a polynomial of the form
$$\Phi_P(1,\dd)=\delta_0+\delta_1 \dd + \cdots + \delta_{\lfloor \frac n 2 \rfloor} \dd ^{\lfloor \frac n 2 \rfloor},$$
where $\lfloor \frac n 2 \rfloor$ is the integer part of $\frac n 2$,
such that each $\delta_i$ is a non-negative integer.
In other words, $\delta_i$ is the sum of coefficients of monomials in $\Phi_P(\cc,\dd)$ for which $\dd$ appears $i$ times.

Let $\Delta$ be a (finite abstract) simplicial complex on the vertex set $V$.
We say that $\Delta$ is \emph{$k$-colored} if there is a map $c: V \to [k]=\{1,2,\dots,k\}$, called a \textit{$k$-coloring map of  $\Delta$}, such that
if $\{x,y\}$ is an edge of $\Delta$ then $c(x) \ne c(y)$.
Let $f_{i}(\Delta)$ denote the number of elements $F \in \Delta$ having cardinality $i+1$,
where $f_{-1}(\Delta)=1$.
The main result of this paper is the following.

\begin{theorem}
\label{main}
Let $P$ be an $(n-1)$-dimensional S*-shellable regular CW-sphere, and let $\Phi_P(1,\dd)=\delta_0+\delta_1 \dd + \cdots + \delta_{\lfloor \frac n 2 \rfloor} \dd ^{\lfloor \frac n 2 \rfloor}$.
Then there exists an $\lfloor \frac n 2 \rfloor$-colored simplicial complex $\Delta$ such that
$$\delta_i = f_{i-1}(\Delta) \ \ \mbox{ for } i=0,1,\dots,\lfloor \frac n 2 \rfloor.$$
\end{theorem}

The precise definition of the S*-shellability is given in Section 2.
The most important class of S*-shellable CW-spheres are the boundary complexes of polytopes.
By the Kruskal-Katona Theorem (see e.g. \cite[II, Theorem 2.1]{St2}), the above theorem gives certain upper bound on $\delta_i$ in terms of $\delta_{i-1}$. Better upper bounds are given by Frankl-F{\"u}redi-Kalai theorem which characterizes the $f$-vectors of $k$-colored complexes \cite{FFK}.

The numbers $\delta_0,\delta_1,\delta_2,\dots$ relate to the $\gamma$-vector (see Section 4 for the definition)
of the barycentric subdivision (order complex) of $P$,
namely the simplicial complex whose elements are the chains of nonempty cells in $P$ ordered by inclusion.
Indeed, as an application of Theorem \ref{main} we prove the following.

\begin{theorem}
\label{second}
Let $P$ be an $(n-1)$-dimensional S*-shellable regular CW-sphere and let $\sd (P)$ be the barycentric subdivision of $P$.
Then there exists an $\lfloor \frac n 2 \rfloor$-colored simplicial complex $\Gamma$ such that
$$\gamma_i(\sd(P))=f_{i-1}(\Gamma) \ \ \mbox{ for } i=0,1,\dots,\lfloor \frac n 2 \rfloor.$$
\end{theorem}

Recall that an $(n-1)$-dimensional simplicial complex is said to be \emph{balanced} if it is $n$-colored.
If $P$ is the boundary complex of an arbitrary convex $n$-dimensional polytope, then
$\delta_{\lfloor \frac n 2 \rfloor}(P)>0$ and we conclude the following.
\begin{corollary}\label{cor:gammaPolytope}
Let $P$ be the boundary complex of an $n$-dimensional polytope. Then
the $\gamma$-vector of $\sd(P)$ is the $f$-vector of a balanced simplicial complex.
\end{corollary}

The above corollary supports the conjecture of Nevo and Petersen \cite[Conjecture 6.3]{NP}
which states that the $\gamma$-vector of a flag homology sphere is the $f$-vector of a balanced simplicial complex.
This conjecture was verified for the barycentric subdivision of simplicial homology spheres (in this case all the cells are simplices) in \cite{Nevo-Petersen-Tenner}.

It would be natural to ask if the above theorems hold for all regular CW-spheres (or more generally, Gorenstein$^*$ posets).
We conjecture a stronger statement on the \cd-index, see Conjecture \ref{conj:flag f-vector}.

This paper is organized as follows:
in Section \ref{sec:2} we recall some known results on the $\cd$-index and define S*-shellability, in Section \ref{sec:3} we prove our main theorem, Theorem \ref{main}, in Section \ref{sec:4} we derive consequences for $\gamma$-vectors and present a conjecture on the \cd-index, Conjecture \ref{conj:flag f-vector}.

\section{\cd-index of S*-shellable CW-spheres}\label{sec:2}

In this section we recall some known results on the \cd-index.

Let $P$ be a graded poset of rank $n+1$ with the minimal element $\hat 0$ and the maximal element $\hat 1$.
Let $\rho$ denote the rank function of $P$.
For $S \subset [n]=\{1,2,\dots,n\}$,
a chain $\hat 0 = \sigma_0 < \sigma_1 <\sigma_2 < \cdots < \sigma_{k+1}=\hat 1$ of $P$
is called an \textit{$S$-flag} if $\{\rho(\sigma_1),\dots,\rho(\sigma_k)\} =S$.
Let $f_S(P)$ be the number of $S$-flags of $P$.
Define $h_S(P)$ by
$$h_S(P)=\sum_{T \subset S} (-1)^{|S|-|T|} f_T (P),$$
where $|X|$ denotes the cardinality of a finite set $X$.
The vectors $(f_S(P): S \subset [n])$ and $(h_S(P): S \subset [n])$ are
called the \textit{flag $f$-vector}  and \textit{flag $h$-vector} of $P$ respectively.

Now we recall the definition of the \cd-index.
For $S \subset [n]$,
we define a non-commutative monomial $u_S=u_1u_2\cdots u_n$ in variables $\aaa$ and $\bb$
by $u_i=\aaa$ if $i \not \in S$ and $u_i=\bb$ if $i \in S$.
Let
$$\Psi_P(\aaa,\bb)= \sum_{S \subset [n]} h_P(S) u_S.$$
For a graded poset $P$,
let $\mathrm{sd}(P)$ be the order complex of $P -\{\hat 0,\hat 1\}$.
Thus
$$\mathrm{sd}(P)=\{\{\sigma_1,\sigma_2,\dots,\sigma_k\} \subset P-\{\hat 0,\hat 1\}: \sigma_1< \sigma_2 < \cdots < \sigma_k\}.$$
We say that $P$ is {\em Gorenstein*} if the simplicial complex $\mathrm{sd}(P)$ is a homology sphere.
It is known that if $P$ is Gorenstein* then $\Psi_P(\aaa,\bb)$ can be written as a polynomial $\Phi_P(\cc,\dd)$ in $\cc=\aaa+\bb$
and $\dd=\aaa \bb + \bb \aaa$ \cite{BK},
and this non-commutative polynomial $\Phi_P(\cc,\dd)$ is called the \textit{$\cc\dd$-index} of $P$.
Moreover,
by the celebrated results due to Stanley \cite{St1} (for convex polytopes)
and Karu \cite{Ka} (for Gorenstein* posets),
the coefficients of $\Phi_P(\cc,\dd)$ are non-negative integers.

Next, we define S*-shellability of regular CW-spheres by slightly modifying the definition of S-shellability
introduced by Stanley \cite[Definition 2.1]{St1}.

Let $P$ be a regular CW-sphere (a regular CW-complex which is homeomorphic to a sphere)
and $\F(P)$ its face poset.
Then the order complex of $\F(P)$ is a triangulation of a sphere,
so the poset $\F(P)\cup\{\hat0,\hat 1\}$ is Gorenstein*.
We define the \cd-index of $P$ by $\Phi_P(\cc,\dd)=\Phi_{\F(P) \cup \{\hat 0,\hat 1\}}(\cc,\dd)$.
For any cell $\sigma$ of $P$,
we write $\bar \sigma$ for the closure of $\sigma$.
For an $(n-1)$-dimensional regular CW-sphere $P$,
let $\Sigma P$ be the suspension of $P$,
in other words, $\Sigma P$ is the $n$-dimensional regular CW-sphere obtained from $P$
by attaching two $n$-dimensional cells $\tau_1$ and $\tau_2$ such that $\partial \bar \tau_1= \partial \bar \tau_2=P$.
Also, for an $(n-1)$-dimensional regular CW-ball $P$
(a regular CW-complex which is homeomorphic to an $(n-1)$-dimensional ball),
let $P'$ be the $(n-1)$-dimensional regular CW-sphere which is obtained from $P$
by adding an $(n-1)$-dimensional cell $\tau$ so that $\partial \bar \tau = \partial P$.

\begin{definition}
\label{def:S*-shellability}
Let $P$ be an $(n-1)$-dimensional regular CW-sphere.
We say that $P$ is \textit{S*-shellable} if either $P=\{\emptyset\}$
or there is an order $\sigma_1,\sigma_2,\dots,\sigma_r$ of the facets of $P$ such that
the following conditions hold.
\begin{itemize}
\item[(a)]
$\partial \bar \sigma_1$ is S*-shellable.
\item[(b)] For $1 \leq i \leq r-1$, let
$$\Omega_i=\bar \sigma_1 \cup \bar \sigma_2 \cup \dots \cup \bar \sigma_i$$
and for $2 \leq i \leq r-1$ let
$$\Gamma_i=\overline{\left[ \partial \bar \sigma_i \setminus \big(\partial \bar \sigma_i \cap \Omega_{i-1}\big) \right]}.$$
Then both $\Omega_i$ and $\Gamma_i$ are regular CW-balls of dimension $(n-1)$ and $(n-2)$ respectively,
and $\Gamma'_i$
is S*-shellable with the first facet of the shelling being the facet which is not in $\Gamma_i$.
\end{itemize}
\end{definition}

\begin{remark}\label{rem:S*-shellability}
The difference between the above definition and Stanley's S-shellability is that S-shellability only assume that
$P$ and $\Gamma'_i$ are Eulerian and assume no conditions on $\Omega_i$.
However, S*-shellable regular CW-spheres are S-shellable, and the boundary complex of convex polytopes are S*-shellable by the line shelling \cite{BM}. We leave the verification of this fact to the readers.

\end{remark}

The next recursive formula is due to Stanley \cite{St1}.

\begin{lemma}[Stanley]
\label{2-3}
With the same notation as in Definition \ref{def:S*-shellability},
for $i=1,2,\dots,r-2$, one has
$$\Phi_{\Omega_{i+1}'}(\cc,\dd) =\Phi_{\Omega_{i}'}(\cc,\dd)
+ \left\{ \Phi_{\Gamma_{i+1}'}(\cc,\dd) -\Phi_{\Sigma (\partial \Gamma_{i+1})}(\cc,\dd)
\right\} \cc
+ \Phi_{\partial \Gamma_{i+1}}(\cc,\dd)\dd.$$
\end{lemma}

Since $\Omega_{r-1}'=P$ the above formula gives a way to compute the \cd-index of $P$ recursively.

Next, we recall a result of Ehrenborg and Karu proving that
the \cd-index increases by taking subdivisions.
Let $P$ and $Q$ be regular CW-complexes,
and let $\phi: \F(P) \to \F(Q)$ be a poset map.
For a subcomplex $Q'= \sigma_1\cup \cdots \cup \sigma_s \subset Q$,
where each $\sigma_i$ is a cell of $Q$,
we write $\phi^{-1}(Q')=\phi^{-1}(\sigma_1) \cup \cdots \cup \phi^{-1}(\sigma_s)$.

Following \cite[Definition 2.6]{EK},
for $(n-1)$-dimensional regular CW-spheres $P$ and $\hat P$,
we say that $\hat P$ is a subdivision of $P$ if there is an order preserving surjective poset map $\phi: \F (\hat P) \to \F (P)$,
satisfying that
for any cell $\sigma$ of $P$, $\phi^{-1}(\bar \sigma)$ is a homology ball having the same dimension as $\sigma$
and $\phi^{-1}(\partial \bar \sigma) = \partial(\phi^{-1}(\bar \sigma))$.


The following result was proved in \cite[Theorem 1.5]{EK}.

\begin{lemma}[Ehrenborg-Karu]
\label{2-4}
Let $P$ and $\hat P$ be $(n-1)$-dimensional regular CW-spheres.
If $\hat P$ is a subdivision of $P$ then one has a coefficientwise inequality $\Phi_{\hat P}(\cc,\dd) \geq \Phi_P(\cc,\dd)$
\end{lemma}

Back to S*-shellable regular CW-spheres,
with the same notation as in Definition \ref{def:S*-shellability},
$\Omega_i'$ is a subdivision of $\Sigma(\partial \Omega_i)$ and
$\partial \Omega_i$ is a subdivision of $\Sigma (\partial \Gamma_{i+1})$.
Indeed, for the first statement, if $\tau_1$ and $\tau_2$ are the facets of $\Sigma(\partial \Omega_i)$
then define $\phi:\F(\Omega_i') \to \F (\Sigma(\partial \Omega_i))$ by
\begin{eqnarray*}
\phi(\sigma)=
\left\{
\begin{array}{llll}
\sigma,& \mbox{ if } \sigma \in \partial \Omega_i,\\
\tau_1, & \mbox{ if } \sigma \mbox{ is an interior face of } \Omega_i,\\
\tau_2, & \mbox{ if } \sigma \not \in \Omega_i.
\end{array}
\right.
\end{eqnarray*}
Similarly, for the second statement, if $\tau_1$ and $\tau_2$ are the facets of $\Sigma(\partial \Gamma_{i+1})$
then define $\phi: \F(\partial \Omega_i) \to \F (\Sigma (\partial \Gamma_{i+1}))$ by
\begin{eqnarray*}
\phi(\sigma)=
\left\{
\begin{array}{llll}
\sigma,& \mbox{ if } \sigma \in \partial \Gamma_{i+1},\\
\tau_1, & \mbox{ if } \sigma \in \bar \sigma_{i+1}\backslash \partial \Gamma_{i+1},\\
\tau_2, & \mbox{ otherwise.}
\end{array}
\right.
\end{eqnarray*}
Since $\Phi_{\Sigma P}(\cc,\dd)=\Phi_P(\cc,\dd) \cc$ for any regular CW-sphere $P$
(see \cite[Lemma 1.1]{St1}), Lemma \ref{2-4} shows

\begin{lemma}
\label{2-5}
With the same notation as in Definition \ref{def:S*-shellability}, for $i=2,3,\dots,r-2$, one has
$\Phi_{\Omega_i'}(\cc,\dd) \geq \Phi_{\partial \Gamma_{i+1}}(\cc,\dd) \cc^2$.
\end{lemma}

\section{Proof of the main theorem}\label{sec:3}

In this section, we prove Theorem \ref{main}.

For a homogeneous $\cd$-polynomial $\Phi$ (i.e., homogeneous polynomial of $\ZZ\langle \cc,\dd\rangle$ with $\deg \cc=1$ and $\deg \dd =2$) of degree $n$,
we define $\Phi_0,\Phi_2,\dots,\Phi_n$ by
$$
\Phi=\Phi_0+ \Phi_2 \dd \cc^{n-2} + \Phi_3\dd\cc^{n-3} + \cdots + \Phi_{n-1} \dd\cc+ \Phi_n \dd
$$
where $\Phi_0=\alpha \cc^n$ for some $\alpha \in \ZZ$ and each $\Phi_k$ is a $\cc\dd$-polynomial of degree $k-2$ for $k \geq 2$.
Also, we write $\Phi_{\leq k}= \Phi_0+ \Phi_2 \dd \cc^{n-2} + \cdots +\Phi_k\dd\cc^{n-k}$.

\begin{definition}\label{def:cd-k-FFK}\
\begin{itemize}
\item
A vector $(\delta_0,\delta_1,\dots,\delta_s) \in \ZZ^{s+1}$ is said to be \textit{$k$-FFK}
if there is a $k$-colored simplicial complex $\Delta$ such that $\delta_i=f_{i-1}(\Delta)$
for $i=0,1,\dots,s$. ($\{\emptyset\}$ is a $0$-colored simplicial complex.)
A homogeneous $\cc \dd$-polynomial $\Phi=\Phi(\cc,\dd)$ is said to be $k$-FFK
if, when we write $\Phi(1,\dd)=\delta_0 + \delta_1 \dd+ \cdots + \delta_s \dd^s$,
the vector $(\delta_0,\delta_1,\dots,\delta_s)$ is $k$-FFK.
\item
A homogeneous $\cc\dd$-polynomial $\Phi$ of degree $n$ is said to be \textit{primitive}
if the coefficient of $\cc^n$ in $\Phi$ is $1$.
\item
Let $\Phi$ be a homogeneous $\cc\dd$-polynomial.
A primitive homogeneous $\cc\dd$-polynomial $\Psi$ is said to be \textit{$k$-good for $\Phi$}
if $\Psi$ is $k$-FFK and $\Phi(1,\dd) \geq \Psi(1,\dd)$.
Also, we say that a homogeneous $\cc\dd$-polynomial $\Psi$ is $k$-good for $\Phi$
if it is the sum of primitive homogeneous $\cc\dd$-polynomials that are $k$-good for $\Phi$.
\end{itemize}
\end{definition}

We will use the following observation, which follows from \cite[Lemma 3.1]{Nevo-Petersen-Tenner}:
\begin{lemma}\label{rem:k-FFK}\
If $\Phi$ is a $k$-FFK homogeneous $\cc\dd$-polynomial of degree $n$,
and if $\Psi'$ and $\Psi''$ are homogeneous $\cc\dd$-polynomials of degree $n'$ and $n''$ respectively,
where $n',n'' \leq n-2$,
which are $k$-good for $\Phi$ then
$$\Phi+ \Psi' \dd\cc^{n-n'-2}\ \rm{and} \ \  \Phi+ \Psi' \dd\cc^{n-n'-2}+\Psi'' \dd\cc^{n-n''-2}$$
are $(k+1)$-FFK.
\end{lemma}
\begin{proof}

By Frankl-F{\"u}redi-Kalai theorem \cite{FFK},
for any $k$-colored simplicial complex $\Gamma$, there is the unique $k$-colored simplicial complex $\mathcal{C}(\Gamma)$, called a $k$-colored compressed complex, such that $f_i(\Gamma)=f_i(\mathcal{C}(\Gamma))$ for all $i$.
Moreover,
if $\Gamma'$ is a $k$-colored complex satisfying $f_i(\Gamma) \leq f_i(\Gamma')$ for all $i$, then one has
$\mathcal{C}(\Gamma) \subset \mathcal{C}(\Gamma')$.

For a simplicial complex $\Gamma$, we write $f(\Gamma,\dd)=1+f_0(\Gamma)\dd+f_1(\Gamma)\dd^2+ \cdots$.
There are $k$-colored complexes $\Delta, \Delta^{(1)},\cdots,\Delta^{(m)},\cdots,\Delta^{(s)}$ such that
$f(\Delta, \dd)=\Phi(1,\dd)$, $\sum_{1\leq i\leq m}f(\Delta^{(i)}, \dd)=\Psi'(1,\dd)$,
$\sum_{m+1\leq i\leq s}f(\Delta^{(i)}, \dd)=\Psi''(1,\dd)$
and each $\Delta^{(i)}$ is a subcomplex of $\Delta$.
Let
$$\Gamma^{(i)}= \Delta \bigcup \left\{\bigcup_{k=1}^i \{ F \cup \{v_k\}: F \in \Delta^{(k)}\}\right\},$$
where $v_1,\dots,v_s$ are new vertices.
Since each $\Delta^{(k)}$ is a subcomplex of $\Delta$,
$\Gamma^{(i)}$ is a simplicial complex.
Also,
$f(\Gamma^{(m)},\dd)=(\Phi+ \Psi' \dd\cc^{n-n'-2})(1,\dd)$ and
$f(\Gamma^{(s)},\dd)=(\Phi+ \Psi' \dd\cc^{n-n'-2} + \Psi'' \dd\cc^{n-n''-2})(1,\dd)$.
We claim that each $\Gamma^{(i)}$ is $(k+1)$-colored.
Let $V$ be the vertex set of $\Delta$ and $c:V\rightarrow [k]$ a $k$-coloring map of $\Delta$.
Then the map $\hat c : V \cup\{v_1,\dots,v_i\} \to [k+1]$ defined
by $\hat c(x)=c(x)$ if $x \in V$ and $\hat c(x) = k+1$ if $x \not \in V$ is a $(k+1)$-coloring map of $\Gamma^{(i)}$.
\end{proof}

Let $P$ be an $(n-1)$-dimensional $S$*-shellable regular CW-sphere with the shelling $\sigma_1,\dots,\sigma_r$.
Keeping the notation in Definition \ref{def:S*-shellability},
to simplify notations, we use the following symbols.
\begin{eqnarray*}
\Phi^{(i)}&=& \Phi^{(i)}(\cc,\dd)= \Phi_{\Omega_i'}(\cc,\dd)\\
\Phi &=& \Phi_P(\cc,\dd)=\Phi^{(r-1)}\\
\Psi^{(i)} &=& \Phi_{\Gamma_{i+1}'}(\cc,\dd)-\Phi_{\Sigma(\partial \Gamma_{i+1})}(\cc,\dd)\\
\Psi &=& \sum_{i=1}^{r-2} \Psi^{(i)}\\
\Pi &=& \Phi-\Phi^{(1)}.
\end{eqnarray*}
Thus Stanley's recursive formula, Lemma \ref{2-3}, says
$$\Phi^{(i+1)}= \Phi^{(i)}+ \Psi^{(i)} \cc + \Phi_{\partial \Gamma_{i+1}}\dd$$
and
$$\Pi= \Psi\cc+ \sum_{i=1}^{r-2} \Phi_{\partial \Gamma_{i+1}}(\cc,\dd) \dd.$$

The last part of the following proposition is a restatement of Theorem \ref{main}.
\begin{proposition}\label{prop:main}
With notation as above, the following holds.
\begin{itemize}
\item[(1)] For $2\leq k\leq n$, $\Psi^{(i)}_k$ is $\lfloor \frac k 2 -1\rfloor$-good for $\Phi_{\leq k-2}^{(i)} + \Psi^{(i)}_{\leq k-2} \cc$.
\item[(2)] For $2\leq k\leq n$, $\Pi_k$ is $\lfloor \frac k 2 -1\rfloor$-good for $\Phi_{\leq k-2}^{(1)} + \Pi_{\leq k-2}$.
\item[(3)] For $2\leq k\leq n$, $\Phi_k$ is $\lfloor \frac k 2 -1\rfloor$-good for $\Phi_{\leq k-2}$.
\item[(4)] For $0\leq k\leq n$, $\Phi_{\leq k}$ is $\lfloor \frac k 2\rfloor$-FFK.
In particular, the $\cc\dd$-index of $P$ is $\lfloor \frac n 2\rfloor$-FFK.
\end{itemize}
\end{proposition}

\begin{proof}
The proof is by induction on dimension, where all statements clearly hold for $n=0,1$.
Suppose that all statements are true up to dimension $n-2$.
To simplify notations, for a regular CW-sphere $Q$,
we write $\Phi_Q=\Phi_Q(\cc,\dd)$.

\emph{Proof of} (1).
By applying the induction hypothesis to $\Gamma_{i+1}'$ (use statement(2)),
each $\Psi_k^{(i)}$ is $\lfloor \frac k 2 -1\rfloor$-good for $(\Phi_{\Sigma(\partial \Gamma_{i+1})})_{\leq k-2}^{(i)} + \Psi^{(i)}_{\leq k-2}$.
Since $(\Upsilon \cc)_k = \Upsilon_k$ for any homogeneous $\cc\dd$-polynomial $\Upsilon$,
$\Psi_k^{(i)}$ is $\lfloor \frac k 2 -1\rfloor$-good for $(\Phi_{\Sigma(\partial \Gamma_{i+1})})_{\leq k-2}^{(i)} \cc + \Psi^{(i)}_{\leq k-2}\cc$.
By Lemma 2.5,
$$\Phi_{\Sigma(\partial \Gamma_{i+1})}\cc = \Phi_{\partial \Gamma_{i+1}}\cc^2
\leq \Phi_{\Omega_{i}'}= \Phi^{(i)},$$
thus
$\Psi^{(i)}_k$ is $\lfloor \frac k 2 -1\rfloor$-good for $\Phi_{\leq k-2}^{(i)} + \Psi^{(i)}_{\leq k-2} \cc$.

\emph{Proof of} (2).
By the definition of $\Pi$,
$$\Pi_k= \sum_{i=1}^{r-1} \Psi^{(i)}_k \mbox{ for }k <n$$
and
$$\Pi_n= \sum_{i=1}^{r-2} \Phi_{\partial \Gamma_{i+1}}.$$
By (1), each $\Psi_k^{(i)}$ is
$\lfloor \frac k 2 -1\rfloor$-good for $\Phi_{\leq k-2}^{(i)} + \Psi^{(i)}_{\leq k-2} \cc$.
Then since
$$\Phi_{\leq k-2}^{(i)} + \Psi^{(i)}_{\leq k-2} \cc \leq \Phi_{\leq k-2} = \Phi_{\leq k-2}^{(1)} + \Pi_{\leq k-2},$$
$\Pi_k$ is $\lfloor \frac k 2 -1\rfloor$-good for $\Phi_{\leq k-2}^{(1)} + \Pi_{\leq k-2}$
for $k <n$.
Also,
each $\Phi_{\partial \Gamma_{i+1}}$ is $\lfloor \frac n 2 -1\rfloor$-FFK
by the induction hypothesis (use (4)),
and $\Phi_{\partial \Gamma_{i+1}} \cc^2 \leq \Phi^{(i)}$ by Lemma 2.5.
The latter condition clearly says
$$\Phi_{\partial \Gamma_{i+1}} \cc^2 \leq \Phi_{\leq n-2}^{(i)} \leq \Phi_{\leq n-2}
= \Phi^{(1)}_{\leq n-2} + \Pi_{\leq n-2}.$$
Hence $\Pi_n$ is $\lfloor \frac n 2 -1\rfloor$-good for $\Phi_{\leq n-2}^{(1)} + \Pi_{\leq n-2}$.

\emph{Proof of} (3).
Observe that since $\Phi^{(1)}= \Phi_{\partial \bar \sigma_1}\cc$,
$$\Phi_k = \Phi_k^{(1)} + \Psi_k \mbox{ for }k<n$$
and
$$\Phi_n = \Pi_n.$$
We already proved that $\Phi_n=\Pi_n$ is $\lfloor \frac n 2 -1\rfloor$-good for $\Phi_{\leq n-2}$ in the proof of (2).
Suppose $k<n$.
Since $\Phi^{(1)}=\Phi_{\partial \bar \sigma_1} \cc$,
by the induction hypothesis (use (3)),
$\Phi_k^{(1)}$ is $\lfloor \frac k 2 -1\rfloor$-good for $\Phi_{\leq k-2}^{(1)}$.
Since $\Phi^{(1)}_{\leq k-2} \leq \Phi_{\leq k-2}$
and since we already proved that $\Psi_k=\Pi_k$ is $\lfloor \frac k 2 -1\rfloor$-good for $\Phi_{\leq k-2}$ in the proof of (2),
$\Phi_k$ is $\lfloor \frac k 2 -1\rfloor$-good for $\Phi_{\leq k-2}$.

\emph{Proof of} (4).
This statement easily follows from (3).
For $k=0,1$, the statement is obvious (as $\Phi_{\leq 0}=\Phi_{\leq 1}=\cc^n$).
Suppose that $\Phi_{\leq 2m+1}$ is $m$-FFK, where $m \in \ZZ_{\geq 0}$.
Then both $\Phi_{2m+2}$ and $\Phi_{2m+3}$ are $m$-good for $\Phi_{\leq 2m+1}$ by (3),
and therefore $\Phi_{\leq 2m+2}$ and $\Phi_{\leq 2m+3}$ are $(m+1)$-FFK
by Lemma \ref{rem:k-FFK}.
\end{proof}
\section{$\gamma$-vectors of polytopes and a conjecture on the \cd-index}\label{sec:4}

\subsection*{$\gamma$-vectors and the \cd-index}

Let $\Delta$ be an $(n-1)$-dimensional simplicial complex.
Then the $h$-vector $h(\Delta)=(h_0,h_1,\dots,h_n)$ of $\Delta$ is defined by the relation
$$\sum_{i=0}^n h_ix^{n-i} = \sum_{i=0}^n f_{i-1}(\Delta) (x-1)^{n-i}.$$
If $\Delta$ is a simplicial sphere (that is, a triangulation of a sphere), or more generally a homology sphere, then $h_i=h_{n-i}$ for all $i$ by the Dehn-Sommerville equations,
and in this case the $\gamma$-vector $(\gamma_0,\gamma_1,\dots,\gamma_{\lfloor \frac n 2 \rfloor})$
of $\Delta$ is defined by the relation
$$\sum_{i=0}^n h_i x^{i} = \sum_{i=0}^{\lfloor \frac n 2 \rfloor} \gamma_i x^i(1+x)^{n-2i}.$$
It was conjectured by Gal \cite{Ga} that if $\Delta$ is a flag homology sphere then its $\gamma$-vector is non-negative.
Recently Nevo and Peterson \cite{NP} further conjectured that the $\gamma$-vector of a flag homology sphere is the $f$-vector of a balanced simplicial complex.
These conjectures are open in general,
the latter conjecture was verified for barycentric subdivisions of simplicial homology spheres \cite{Nevo-Petersen-Tenner},
and Gal's conjecture is known to be true for barycentric subdivisions of regular CW-spheres by the following fact, combined with Karu's result on the nonnegativity of the $\cd$-index for Gorenstien$^*$ posets:

Let $P$ be an $(n-1)$-dimensional regular CW-sphere.
The \textit{barycentric subdivision} $\mathrm{sd}(P)$ of $P$ is the order complex of $\F(P)$.
Let $(h_0,h_1,\dots,h_n)$ and $(\gamma_0,\gamma_1,\dots,\gamma_{\lfloor \frac n 2 \rfloor})$ be the $h$-vector and $\gamma$-vector of $\sd(P)$, respectively.
Then it is easy to see that $h_i= \sum_{S \subset [n],\ |S|=i} h_S(P)$.
Thus if $\Phi_P(1,\dd)= \delta_0+ \delta_1 \dd + \delta_2  \dd^2 + \cdots+ \delta_{\lfloor \frac n 2 \rfloor} \dd^{\lfloor \frac n 2 \rfloor}$, then for all $i\geq 0$,
$$\gamma_i=2^i \delta_i.$$
Since $\delta_i$ is non-negative, we conclude that $\gamma_i$ is also non-negative.

The next simple statement, combined with Theorem \ref{main}, proves Theorem \ref{second}.
\begin{lemma}
\label{4-1}
With the same notation as above,
if $(\delta_0,\delta_1,\dots,\delta_{\lfloor \frac n 2 \rfloor})$ is $k$-FFK then $(\gamma_0,\gamma_1,\dots,\gamma_{\lfloor \frac n 2 \rfloor})$ is also $k$-FFK.
\end{lemma}

\begin{proof}
Let $\Delta$ be a $k$-colored simplicial complex on the vertex set $V$ with $f_{i-1}(\Delta)=\delta_i$ for all $i \geq 0$ and let $c: V\rightarrow [k]$ be a $k$-coloring map of $\Delta$.
Consider a collection of subsets of $W=\{x_v:v \in V\}\cup \{y_v: v \in V\}$
$$\hat \Delta =\{x_G \cup y_{F \setminus G}:
F \in \Delta,\ G \subset F\},$$
where $x_H=\{x_v: v\in H\}$ and $y_H=\{y_v: v\in H\}$ for any $H \subset V$.
Then $\hat \Delta$ is a simplicial complex with $f_{i-1}(\hat \Delta)=2^i f_{i-1}(\Delta)=\gamma_i$ for all $i$.
The map $\hat c : W\rightarrow [k]$, $\hat c (x_v)=\hat c (y_v) = c(v)$, shows that $\hat \Delta$ is $k$-colored.
\end{proof}

\begin{proof}[Proof of Corollary \ref{cor:gammaPolytope}]
By Theorem \ref{second}, in order to prove Corollary \ref{cor:gammaPolytope} it is enough to show that $\delta_{\lfloor \frac n 2 \rfloor}(P)>0$ where $P$ is the boundary complex of an $n$-polytope.
Billera and Ehrenborg showed that the $\cd$-index of $n$-polytopes is minimized (coefficientwise) by the $n$-simplex, denoted $\sigma^n$ \cite{BE}. Thus, it is enough to verify that
$\delta_{\lfloor \frac n 2 \rfloor}(\sigma^n)>0$. It is known that \emph{all} the $\cd$-coefficients of $\sigma^n$ are positive (e.g., by using the Ehrenborg-Readdy formula for the $\cd$-index of a pyramid over a polytope \cite[Theorem 5.2]{ER}).
\end{proof}

%

\subsection*{A conjecture on the \cd-index}
It would be natural to ask if Theorems \ref{main} and \ref{second} hold for all regular CW-spheres (or all Gorenstein* posets).
We phrase a conjecture on the the $\cd$-index, that, if true, immediately implies Theorem \ref{main}, as well as
the entire Proposition \ref{prop:main}(4).

For an arbitrary $\cd$-monomial $w=\cc^{s_0} \dd \cc^{s_{1}} \dd \cdots \dd \cc^{s_k}$ of degree $n$ (where $0\leq s_i$ for all $i$ and $s_0+\dots+s_k+2k=n$), let $F_w$ be the following subset of $[n-1]$:
$$F_w=\{s_0+1,s_0+s_1+3,s_0+s_1+s_2+5,\dots,s_0+\dots+s_{k-1}+2k-1 \}.$$
Note that $F_w$ contains no two consecutive numbers.  For example, $F_{\cc^n}=\emptyset$, $F_{\dd^{k}}=\{1,3,\dots,2k-1\}$ and
 $F_{\cc\dd^{k}}=\{2,4,\dots,2k\}$.
Let $\A$ be the set of subsets of $[n-1]$ that have no two
consecutive numbers, and let $\B$ be the set of $\cd$-monomials of degree $n$. Then $w\mapsto F_w$ is a bijection from $\B$ to $\A$ (as $k=|F_w|$ and $s_k=n-2k-s_{k-1}-\dots -s_0$ we see that the inverse map exists).


Let $\Delta$ be a $k$-colored simplicial complex with the vertex set $V$ and a $k$-coloring map $c: V \to [k]$.
For any subset $S \subset [k]$, let $f_S(\Delta)=|\{F \in \Delta: \ c(F)=S\}|$.
The vector $(f_S(\Delta): S \subset [k])$ is called the \textit{flag $f$-vector of $\Delta$}.
Note that the flag $f$-vector of a Gorenstein* poset $P$ is equal to the flag $f$-vector of $\sd(P)$
by the coloring map defined by the rank function.

\begin{definition}\label{def:cd-flag-f-vector}\
Let $\Phi=\sum_w a_w w$ be a homogeneous $\cd$-polynomial of degree $n$ with $w$ the $\cd$-monomials and $a_w\in \ZZ$.
For $S \subset [n-1]$,
we define
\begin{eqnarray*}
\alpha_S(\Phi)=
\left\{
\begin{array}{ll}
a_w, & \mbox{ if  $S=F_w$ for some $w \in \B$}\\
0, & \mbox{ if $S \not \in \A$}.
\end{array}
\right.
\end{eqnarray*}
\end{definition}


\begin{conjecture}\label{conj:flag f-vector}
Let $P$ be an $(n-1)$-dimensional regular CW-sphere (or more generally, Gorenstein* poset of rank $n+1$).
Then there exists an $(n-1)$-colored simplicial complex $\Delta$ such that $f_S(\Delta)=\alpha_S(\Phi_P)$ for all $S \subset [n-1]$.
\end{conjecture}

Thus the above conjecture states that the \cd-index is itself the flag $f$-vector of a colored complex.
If the above conjecture is true
then $\Phi_P(1,\dd)=1+ f_0(\Delta)\dd + \cdots + f_{\lfloor \frac n 2 \rfloor-1}(\Delta)\dd^{{\lfloor \frac n 2 \rfloor}}$.
Although $\Delta$ is $(n-1)$-colored,
this fact implies Theorem \ref{main}.
Indeed, since $f_S(\Delta)=\alpha_S(\Phi_P)=0$ if $S$ has consecutive numbers,
if $c: V \to [n-1]$ is an $(n-1)$-coloring map of $\Delta$ then the map $\hat c : V \to [\lfloor \frac n 2 \rfloor ]$
defined by $\hat c (v)=\lfloor \frac {c(v) +1} {2}\rfloor$ is an $\lfloor \frac n 2 \rfloor$-coloring map of $\Delta$.

The next result supports the conjecture in low dimension.

\begin{proposition}
\label{lowdimensionalcase}
Let $P$ be a Gorenstein* poset of rank $n+1$. For all $i,j \in [n-1]$,
$$\alpha_{\{i\}}(\Phi_P) \alpha_{\{j\}}(\Phi_P) \geq \alpha_{\{i,j\}}(\Phi_P).$$
\end{proposition}

\begin{proof}
Let $(h_S(P): S \subset [n])$ be the flag $h$-vector of $P$.
Let $\{i,i+j\} \subset [n-1]$ with $j \geq 2$.
What we must prove is
$\alpha_{\{i\}}(\Phi_P) \alpha_{\{i+j\}}(\Phi_P) \geq \alpha_{\{i,i+j\}}(\Phi_P).$

Observe that
\begin{eqnarray*}
h_{[i]\cup\{i+j+1,\dots,n\}}(P)
&=&
\alpha_{\{i,i+j\}}(\Phi_P) + \alpha_{\{i\}}(\Phi_P) + \alpha_{\{i+j\}}(\Phi_P) + \alpha_\emptyset(\Phi_P),\\
h_{[i]}(P)
&=&\alpha_{\{i\}}(\Phi_P) + \alpha_\emptyset(\Phi_P),\\
h_{\{i+j+1,\dots,n\}}(P)
&=&\alpha_{\{i+j\}}(\Phi_P) + \alpha_\emptyset(\Phi_P)
\end{eqnarray*}
(as $h_{[i]\cup\{i+j+1,\dots,n\}}(P)$ is the coefficient of
$\bb^i \aaa^j \bb^{n-i-j}$ in $\Psi_P(\aaa,\bb)$, etc.).
Since $\alpha_{\emptyset}=1$,
it is enough to prove that
$$h_{[i]}(P)h_{\{i+j+1,\dots,n\}}(P)
\geq h_{[n-i-j]\cup\{n-i+1,\dots,n\}}(P).$$
It follows from \cite[III, Theorem 4.6]{St2} that there is an $n$-colored simplicial complex $\Delta$ with a coloring map $c: V \to [n]$ such that
$f_S(\Delta)=h_S(P)$ for all $S \subset [n]$.
Let
$$\Delta_S=\{F \in \Delta: c(F)=S\}$$
for $S \subset [n]$.
Then it is clear that
$$\Delta_{[i]\cup\{i+j+1,\dots,n\}}
\subset \{F \cup G: F \in \Delta_{[i]},\ G \in \Delta_{\{i+j+1,\dots,n\}}\},$$
which implies the desired inequality.
\end{proof}

It is straightforward that the above proposition proves the next statement.

\begin{corollary}
Conjecture \ref{conj:flag f-vector} holds for $n \leq 5$.
\end{corollary}

\subsection*{Non-existence of $\dd$-polynomials}
For a Gorenstein* poset $P$,
we call $\Phi_P(1,\dd)$ \textit{the $\dd$-polynomial of $P$}.
It is a challenging problem to classify all possible $\dd$-polynomials of Gorenstein* posets,
which give a complete characterization of all possible face vectors of Gorenstein* order complexes
since knowing $\dd$-polynomials is equivalent to knowing $\gamma$-vectors.
The problem is open even for the $3$-dimensional case.
To study this problem,
by virtue of Theorem \ref{main}, it is natural to ask which FFK vector is realizable as the $\dd$-polynomial of a Gorenstein* poset.
The next result shows that not all $\lfloor \frac n 2 \rfloor$-FFK vectors are realizable as the $\dd$-polynomial of a Gorenstein* poset of rank $n+1$.

First recall that the ordinal sum $Q_1 + Q_2$ of two disjoint posets $Q_1$ and $Q_2$ is the poset whose elements are the union of elements in $Q_1$ and $Q_2$ and whose relations are those in $Q_1$ union those in $Q_2$ union all $q_1<q_2$ where $q_1\in Q_1$ and $q_2\in Q_2$.
For Gorenstein* posets $Q_1$ and $Q_2$,
the poset
$Q_1 * Q_2=(Q_1-\{\hat 1\}) + (Q_2 -\{\hat 0\})$
is called the \textit{join} of $Q_1$ and $Q_2$,
and $\Sigma Q_1=Q_1 * B_2$, where $B_2$ is a Boolean algebra of rank $2$,
is called the \textit{suspension} of $Q_1$.
By \cite[Lemma 1.1]{St1},
$\Phi_{Q_1*Q_2}(\cc,\dd)=\Phi_{Q_1}(\cc,\dd)\cdot \Phi_{Q_2}(\cc,\dd)$.

\begin{proposition}\label{prop:rank5}
Let $P$ be a Gorenstein* poset of rank $5$,
and let
$$\Phi_P(\cc,\dd)= \cc^4 +  \alpha_{\{1\}} \cc^2\dd +  \alpha_{\{2\}} \cc\dd\cc
 +  \alpha_{\{3\}} \dd\cc^2+  \alpha_{\{1,3\}} \dd^2$$
be its $\cc\dd$-index.
Suppose $\alpha_{\{2\}}=0$.
Then there are Gorenstein* posets $P_1$ and $P_2$ of rank $3$
such that $P=P_1*P_2$.
In particular, $\alpha_{\{1,3\}}=\alpha_{\{1\}}\alpha_{\{3\}}$.
\end{proposition}

\begin{proof}
Let $r$ denote the rank function $r:P\rightarrow \{0,1,\dots,5\}$ ($r(\hat{0})=0, r(\hat{1})=5$).
Let
$P_1:=\{F\in P:\ r(F)\leq 2\}$
and
$P_2:=\{F\in P:\ r(F)\geq 3\}$.

As $P$ is Gorenstien*, to show that
$P=P_1 + P_2 $ it is enough to show that $P_2\cup \{\hat{0}\}$ is Gorenstien* (as a Gorenstien* poset contains no proper subposet which is Gorenstien* of the same rank, and each interval $[F,\hat{1}]$ with $r(F)=2$ in $P$ is Gorenstien*).
For this, it is enough to show that any rank $4$ element in $P$ covers exactly two rank $3$ elements in $P$. Indeed, this guarantees that the dual poset to $P_2$, denoted $P_2^*$, is the face poset of a union of CW $1$-spheres, and as $P$ is Gorenstien* so is its dual $P^*$, hence $P_2^*$ is Cohen-Macaulay
since $P_2^*$ is a rank selected poset \cite[III, Theorem 4.5]{St2}, which implies that $P_2^*$ is the face poset of one CW $1$-sphere, i.e.\ $P_2\cup \{\hat{0}\}$ is Gorenstien*.

Let $F$ be a rank $4$ element of $P$. Then $P$ is a subdivision of $\Sigma ([\hat 0, F])$ (Recalling \cite[Definition 2.6]{EK}, this is shown by the map $\phi:P\rightarrow \Sigma ([\hat 0, F])$, $\phi(\sigma)=\sigma$ if $\sigma < F$, $\phi(\sigma)=\sigma_1$ if $\sigma$ and $F$ are incomparable, and $\phi(F)=\sigma _2$, where $\sigma_1,\sigma_2$ are the rank $4$ elements in  $\Sigma ([\hat 0, F]$).
Thus, by Lemma \ref{2-4}, the coefficient of $\cc\dd\cc$ in the $\cd$-index of  $\Sigma ([\hat 0, F])$ is zero, hence the coefficient of the monomial $\cc\dd$ in the $\cd$-index of $[\hat 0, F]$ is zero.

This fact implies, when expanding the $\cd$-index of $[\hat 0, F]$ in terms of $\aaa,\bb$, that $h_{\{3\}}([\hat 0, F])$ equals the coefficient of $\cc^3$, namely $h_{\{3\}}([\hat 0, F])=1$.
Switching to the flag $f$-vector of $[\hat 0, F]$ we get
$f_{\{3\}}([\hat 0, F])= h_{\emptyset}([\hat 0, F]) + h_{\{3\}}([\hat 0, F]) =1+1=2$.
Thus, $F$ covers exactly two rank $3$ elements in $P$.
\end{proof}

\begin{example}
Consider the $2$-FFK vector $(1,6,7)$.
We claim that $\Phi_P(1,\dd) \ne 1 + 6 \dd + 7\dd^2$ for all Gorenstein$^*$ poset $P$ of rank $5$.
Indeed,
if $\Phi_P(1,\dd)=1 + 6 \dd + 7\dd^2$, then $\alpha_{\{1,3\}}=7$.
Then $\alpha_{\{1\}}+\alpha_{\{3\}}=6$
and $\alpha_{\{2\}}=0$ by Proposition \ref{lowdimensionalcase}, which contradicts Proposition \ref{prop:rank5}.

A similar argument shows that $(1,2a,a^2-2)$, where $a \geq 3$, is $2$-FFK,
but not realizable as the $\dd$-polynomial of a Gorenstein* poset of rank $5$.
\end{example}


\end{document}